%% file: 4.tex
\documentclass[12pt]{article}
\usepackage{graphicx}
\usepackage{amsmath}
\usepackage{amsfonts}
\usepackage{amsthm}
\usepackage[T1]{fontenc}
\usepackage{url}
\usepackage[margin=1in]{geometry}
\input{mymacros}
\usepackage{tikz}
\numberwithin{equation}{section}
\begin{document}
\title{Weighted asymptotic Korn and interpolation Korn inequalities with singular weights}
\author{Davit Harutyunyan\footnote{University of California Santa Barbara, harutyunyan@ucsb.edu} and
Hayk Mikayelyan\footnote{University of Nottingham Ningbo China, Hayk.Mikayelyan@nottingham.edu.cn} }
\maketitle

\begin{abstract}
In this work we derive asymptotically sharp weighted Korn and Korn-like interpolation (or first and a half) inequalities in thin domains with singular weights. The constants $K$ (Korn's constant) in the inequalities depend on the domain thickness $h$ according to a power rule $K=Ch^\alpha,$ where $C>0$ and $\alpha\in R$ are constants independent of $h$ and the displacement field. The sharpness of the estimates is understood in the sense that the asymptotics $h^\alpha$ is optimal as $h\to 0.$ The choice of the weights is motivated by several factors, in particular a spacial case occurs when making Cartesian to polar change of variables in two dimensions.
\end{abstract}

\hspace{-0.65cm}\textbf{Keywords}\ \ Korn inequality; Weighted Korn inequality; thin domains
\newline
\linebreak

\section{Introduction}
\label{sec:1}
Since the pioneering work of Korn [\ref{bib:Korn.1},\ref{bib:Korn.2}], Korn and Korn-like inequalities, such as geometric rigidity estimates [\ref{bib:Fri.Jam.Mue.1},\ref{bib:Fri.Jam.Mue.2}] as well as extensions [\ref{bib:Nef.Pau.Wit.1}] have been known to play a central role in the theories of linear [\ref{bib:Friedrichs},\ref{bib:Pay.Wei.},\ref{bib:Gra.Tru.},\ref{bib:Gra.Har.2},\ref{bib:Gra.Har.3}] and nonlinear [\ref{bib:Fri.Jam.Mue.1},\ref{bib:Fri.Jam.Mue.2}] elasticity. Korn's first inequality has been introduced by Korn [\ref{bib:Korn.1},\ref{bib:Korn.2}] to prove the coercivity of the linear elastic energy, and it asserts the following: \textit{Given $\Omega\subset\mathbb R^n$ and a closed subspace of vector fields $V\subset H^1(\Omega,\mathbb R^n),$ that has a trivial intersection with the subspace $\mathrm{skew}(\mathbb R^n)=\{Ax+b\ : \ A\in \mathbb M^{n\times n}, A^T=-A, b\in \mathbb R^n\},$ of rigid body motions, i.e., $V\cap \mathrm{skew}(\mathbb R^n)=\{0\},$ there exists a constant $C,$ depending only on $\Omega$ and $V,$ such that for any vector field $\Bu\in V$ the inequality holds:
\begin{equation}
\label{1.1}
C\|\nabla \Bu\|_{L^2(\Omega)}^2\leq \|e(\Bu)\|_{L^2(\Omega)}^2.
\end{equation}
Here, $e(\Bu)=\frac{1}{2}\left(\nabla \Bu+\nabla \Bu^T\right)$ is the symmetric part of the gradient, i.e., strain in linear elasticity.}
Korn's second inequality reads as follows:
\textit{Given $\Omega\subset\mathbb R^n,$ there exists a constant $C,$ depending only on $\Omega,$ such that for any vector field
$\Bu\in H^1(\Omega,\mathbb R^n)$ the inequality holds:}
\begin{equation}
\label{1.1}
C\|\nabla \Bu\|_{L^2(\Omega)}^2\leq \|e(\Bu)\|_{L^2(\Omega)}^2+\|\Bu\|_{L^2(\Omega)}^2.
\end{equation}
It has been known that in thin structure, such as rod plate and shell theories, the dependence of the constant $C$ in Korn's inequalities on the geometric parameters of the domain $\Omega$ becomes crucial, e.g., [\ref{bib:Fri.Jam.Mue.1},\ref{bib:Fri.Jam.Mue.2},\ref{bib:Gra.Tru.},\ref{bib:Gra.Har.2},\ref{bib:Gra.Har.3}].
Especially it is important to know how the optimal constant $C$ scales with the thickness $h$ of the thin structure as $h$ goes to zero.
For plates the constant $C$ has been proved to scale like $h^2$ by Friesecke, James and M\"uller [\ref{bib:Fri.Jam.Mue.1}] even in the geometric rigidity estimate, which is the nonlinear analog of Korn's first inequality. When the shell has a nonzero principal curvature, then the scaling $h^2$ is no longer optimal, and new exponents $\alpha$ satisfying $1\leq \alpha\leq 1.5$ occur as shown in [\ref{bib:Gra.Har.1},\ref{bib:Gra.Har.4},\ref{bib:Harutyunyan3}]. We will call such inequalities sharp. The recent survey book chapter by Stefan M\"uller [\ref{bib:Mueller.}] gives a complete picture on the above issues and applications as well as the open problems in the field. 
Sharp Kotn's inequalities for thin structures, such as rods, plates, shells and combinations of those, have been recently studied by several authors and groups. We refer to the works as well as the above mentioned ones and the references therein for more detailed information.  [\ref{bib:Aco.Cej.Dur.},\ref{bib:Aco.Dur.},\ref{bib:Aco.Dur.Lom.},\ref{bib:Aco.Dur.LopGar.},\ref{bib:Naz.1},\ref{bib:Naz.2},\ref{bib:Naz.3},\ref{bib:Naz.Slu.}].
In the present work we deal with weighted Korn and Korn-like inequalities, on which there is relatively less information in the literature 
[\ref{bib:Aco.Cej.Dur.},\ref{bib:Aco.Dur.Lom.},\ref{bib:Aco.Dur.LopGar.},\ref{bib:Harutyunyan2},\ref{bib:LopGar1.}]. The recent work of Lopez Garcia [\ref{bib:LopGar2.}] goes further and establishes a generalization of Korn inequalities in the case when the domain $\Omega$ is not necessarily thin and thus one is not interested in sharp estimates. We refer to the book of Acosta and Duran [\ref{bib:Aco.Dur.}] for a more detailed discussion of the subject and possible applications. Another motivation of ours of considering weighted Korn and Korn-like inequalities is the following: When dealing with radially symmetric structures, it is convenient to make a Cartesian to polar change of variables, where a weight $w^2=r$ occurs in the norms, which vanishes at the origin and thus becomes singular. The Korn and similar inequalities under consideration become weighted ones with the above weight, which do not follow from the non-weighted analogues due to the singularity of the weight. The case of two spatial dimensions and $w^2=r$ is partially studied in [\ref{bib:Harutyunyan2}], and applied to prove optimal Korn inequalities for washers.
Another aspect is that the classical Kotn's first inequality requires a least one condition on the displacement, such as a boundary or a normalization condition, whereas the analogous geometric rigidity estimate does not. Therefore, in order to avoid the imposed boundary conditions, one may be able to apply a localization argument in some parts of the domain, by considering the analogous weighted version of the inequality under consideration. Of course the last is a delicate question and is task for out future studies.

\section{Main results}
\label{sec:2}
We assume in the sequel that $n\in\mathbb N,$ $n\geq 2,$ $\omega\subset\mathbb R^{n-1}$ and $\Omega\subset\mathbb R^n$ are open bounded connected Lipschitz domains. Let the constant matrix $A=\{a_{ij}\}_{i,j=1}^n\in \mathbb M^{n\times n}$ be positive definite with eigenvalues between the positive constants $0<\lambda\leq \Lambda,$ i.e.,
\begin{equation}
\label{2.1}
\lambda |\xi|^2\leq \xi A\xi^T\leq \Lambda |\xi|^2\quad\text{ for all}\quad \xi\in\mathbb R^n.
\end{equation}
Set next the elliptic operator
\begin{equation}
\label{2.2}
L(u)=\mathrm{div}(A\nabla u),\quad \text{for all}\quad u\in H^1(\Omega).
\end{equation}

The following gradient separation estimate for solutions of elliptic equations is one of the main results of the paper. It has been shown that this kind of estimates derive Korn's first and Korn interpolation inequalities with the same weight in two space dimensions [\ref{bib:Gra.Har.1},\ref{bib:Harutyunyan1},\ref{bib:Harutyunyan2}].
\begin{theorem}
\label{thm:2.1}
Let $n\in\mathbb N,$ $n\geq 2,$ let $\omega\subset\mathbb R^{n-1}$ be an open bounded connected Lipschitz domain with the following properties:
there exists a partition of the boundary $\partial\omega=\Gamma_1\cup\Gamma_2$ and a number
$0\leq d\leq n-1,$ such that
\begin{itemize}
\item[(i)] $\Gamma_1$ is a $d-$dimensional simplex.
\item[(ii)] $\omega$ is star shaped with respect to $\Gamma_1,$ i.e., for any points $x\in\Gamma_1$ and $y\in\omega,$ the ray $l_x(y)$ starting
    from $x$ and going through $y$ meets the boundary of $\omega$ second time at $z\in\Gamma_2,$ such that the segment $(x,z)$ is the only common part of $l_x(y)$ and $\omega.$
\end{itemize}
Assume further $h>0,$ and denote $\Omega=(0,h)\times\omega.$ Let the matrix $A\in \mathbb M^{n\times n}$ and the operator
$L(u)=\mathrm{div}(A\nabla u)$ be as in (\ref{2.1}) and (\ref{2.2}). Assume $k\in \mathbb N$ and the exponents $\alpha_i$ and the coefficients $c_i$ satisfy the conditions $\alpha_i\in [0,1/2)$ and $\,c_i\geq 0$ for $i=1,2,\dots,k$.
Denote $\delta(x)=\mathrm{dist}(x,(0,h)\times \Gamma_1)$ and $w(x)=c_1\delta^{\alpha_1}(x)+c_2\delta^{\alpha_2}(x)+\dots+c_k\delta^{\alpha_k}(x)$ for $x\in\Omega.$ If the function $u\in C^2(\Omega)\cap C(\bar\Omega)$ solves the equation $L(u)=0$ in $\Omega$ and satisfies the boundary conditions $u=0$ on $(0,h)\times\partial\omega,$ and the exponent $\beta\in [0,1/2)$ satisfies the bound 
\begin{equation}
\label{2.2.1}
\lambda>\frac{4n\Lambda\beta}{(1-2\beta)^2},
\end{equation} 
then there exists a computable constant $C,$ depending only on the quantities $\lambda, \Lambda, n, k$ and $\beta,$ such that
\begin{equation}
\label{2.3}
\|w\nabla u\|_{L^2(\Omega)}^2\leq C\left(\frac{\|wu\|_{L^2(\Omega)}\cdot\|wu_{x_1}\|_{L^2(\Omega)}}{h}+\|wu_{x_1}\|_{L^2(\Omega)}^2\right),
\end{equation}
whenever $\alpha_i\in [0,\beta],$ for $i=1,2,\dots,k$.
Moreover, if the operator $L$ is the Laplacian and $n=2,$ i.e., $L=\Delta$ and $\omega=(a,b)$ for some $a<b,$ then the estimate (\ref{2.3}) holds true for all values of $\beta\in[0,1/2],$ which means that imposing the condition (\ref{2.2.1}) is not necessary.
\end{theorem}

Next theorem is the analogous Korn's interpolation inequality in two space dimensions.
\begin{theorem}
\label{thm:2.2}
For $a,h>0$ set $R=(0,h)\times(0,a).$ Let $k\in\mathbb N$ and let the exponents $\alpha_i\in \mathbb R$ and the function $w$ be as in Theorem~\ref{thm:2.1}. Assume the displacement $\BU=(u,v)\in H^1(R,\mathbb R^2)$ satisfies the boundary condition
$u(x,0)=u(x,a)$ for all $x\in(0,h)$ in the sense of traces. Then the weighted Korn interpolation inequality holds:
\begin{equation}
\label{2.4}
\|w\nabla \BU\|_{L^2(R)}^2\leq C\left(\frac{\|wu\|_{L^2(\Omega)}\cdot\|we(\BU)\|_{L^2(\Omega)}}{h}+\|we(\BU)\|_{L^2(\Omega)}^2\right).
\end{equation}
\end{theorem}
Finally, we emphasize the sharpness of the estimates (\ref{2.3}) and (\ref{2.4}) in terms of the asymptotics of $h$ as $h\to 0.$

\begin{theorem}
\label{thm:2.3}
For the case $L=\Delta,$ the estimate (\ref{2.3}) is sharp in terms of the asymptotics of $h$ as $h\to 0.$ Also, the estimate (\ref{2.4})
is sharp as well in terms of the asymptotics of $h$ as $h\to 0.$
\end{theorem}

\section{Preliminaries}
\label{sec:3}

\begin{lemma}
\label{lem:3.1}
Assume $\Omega\in\mathbb R^n$ is an open bounded set with Lipschitz boundary and assume $\partial\Omega=\Gamma_1\cup\Gamma_2.$
Denote by $\delta$ the distance function from the $\Gamma_1$ part of the boundary $\partial\Omega,$ i.e., $\delta(x)=\rm{dist}(x,\Gamma_1)$ for $x\in\mathbb R^n.$ Assume the weight function $w\in L^2(\Omega)\cap H_{loc}^1(\Omega)$ is such that
\begin{equation}
\label{3.2}
w(x)\geq 0,\quad|\delta(x)\nabla w(x)|\leq K\cdot w(x)\quad\text{for all}\quad x\in\Omega,
\end{equation}
for some $K>0.$ If the function $u\in H^1(\Omega)$ satisfies the boundary condition $u(x)=0$ for $x\in \Gamma_2$ in the sense of traces, then there exists a constant $C,$ depending only on the quantities $\lambda, \Lambda, K$ and $n$, such that the estimate holds:
\begin{equation}
\label{3.3}
\|w\delta\nabla u\|_{L^2(\Omega)}^2\leq C\left(\|wu\|_{L^2(\Omega)}^2+\|w\delta^2L(u)\|_{L^2(\Omega)}^2\right).
\end{equation}

\end{lemma}

\begin{proof} Let us mention that in the proof the constant $C$ depends only on the quantities $\lambda, \Lambda, K$ and $n.$ Following [\ref{bib:Harutyunyan2}], we evaluate for any $u\in H^1(\Omega)$ using the boundary conditions on $u,$
\begin{equation}
\label{3.4}
I_3=\int_{\Omega}w^2\delta^2uL(u)dx=\int_{\Omega}w^2\delta^2u[\mathrm{div}(A\nabla u)]dx=I+I_1+I_2,
\end{equation}
where
\begin{equation}
\label{3.5}
I=-\int_{\Omega}w^2\delta^2\nabla u\cdot(A\nabla u)dx,
\end{equation}
\begin{equation}
\label{3.6}
I_1=-2\int_{\Omega}w^2\delta u\nabla\delta\cdot( A\nabla u)dx,\quad I_2=-2\int_{\Omega}w\delta^2 u \nabla w\cdot(A\nabla u)dx.
\end{equation}
Observe that as $\delta$ is a distance function, thus it is Lipschitz and weeakly differentiable a.e. with $|\nabla \delta|\leq 1.$ Consequently we have by the geometric-arithmetic mean inequality in the form $2ab\leq \epsilon a^2+\frac{1}{\epsilon}b^2,$ that
\begin{equation}
\label{3.7}
|I_1|\leq \epsilon\int_{\Omega}w^2\delta^2 |\nabla u|^2dx+\frac{C}{\epsilon}\int_{\Omega}w^2u^2dx,
\end{equation}
where $\epsilon>0$ is a number yet to be chosen. We have similarly that
$$
|I_2|\leq \epsilon\int_{\Omega}w^2\delta^2 |\nabla u|^2dx+\frac{C}{\epsilon}\int_{\Omega}\delta^2|\nabla w|^2u^2dx,
$$
thus owing to the bound (\ref{3.2}) we get
\begin{equation}
\label{3.8}
|I_2|\leq \epsilon\int_{\Omega}w^2\delta^2 |\nabla u|^2dx+\frac{CK^2}{\epsilon}\int_{\Omega}w^2u^2dx.
\end{equation}
Finally we have for $I_3$ by the Schwartz inequality, that
\begin{equation}
\label{3.9}
|I_3|\leq \frac{1}{2}\int_{\Omega}w^2u^2dx+\frac{1}{2}\int_{\Omega}w^2\delta^4L(u)^2dx.
\end{equation}
By the positive definiteness condition (\ref{2.1}) of $A$ we get the lower bound
\begin{equation}
\label{3.10}
I=\int_{\Omega}w^2\delta^2\nabla u\cdot(A\nabla u)dx\geq \lambda \int_{\Omega}w^2\delta^2|\nabla u|^2,
\end{equation}
thus combining the estimates (\ref{3.4})-(\ref{3.10}) and choosing $\epsilon=\frac{\lambda}{4},$ we obtain (\ref{3.3}).

\end{proof}

Next we give some useful examples of weights $w$ satisfying the hypothesis (\ref{3.2}). The statement is formulated in the below lemma.
\begin{lemma}
\label{lem:3.2}
Exponents of distance functions such as $w(x)=\left(\mathrm{dist}(x,x_0)\right)^\alpha$ and $w(x)=\left(\mathrm{dist}(x,\Gamma_1)\right)^\alpha$ fulfill the condition (\ref{3.2}) for any point $x_0\in \Gamma_1,$ with $K=|\alpha|.$ If functions $w_1(x),w_2(x),\dots,w_k(x)$ satisfy (\ref{3.2}) with the same constant $K$, then any linear combination $w(x)=\sum_{i=1}^kc_kw_k(x)$ with positive coefficients $c_i\geq 0$ satisfies (\ref{3.2}) with the same constant $K.$ Also, the product weight $w(x)=w_1(x)w_2(x)\cdot\ldots\cdot w_k(x)$ satisfies (\ref{3.2}) with a constant $kK.$
\end{lemma}

\begin{proof}
The proof is elementary. We have for the case $w(x)=\left(\mathrm{dist}(x,\Gamma_1)\right)^\alpha=\delta(x)^\alpha$ that $\nabla w(w)=\alpha\delta(x)^{\alpha-1}\nabla \delta(x),$ thus $|\delta(x)\nabla w(x)|\leq |\alpha|\delta(x)^{\alpha}.$ For the case $w(x)=\left(\mathrm{dist}(x,x_0)\right)^\alpha$ we have thanks to the inequality $\delta(x)\leq \mathrm{dist}(x,x_0),$ that
$$|\delta(x)\nabla w(x)|=|\alpha|\left(\mathrm{dist}(x,x_0)\right)^{\alpha-1}\delta(x)|\nabla\mathrm{dist}(x,x_0)|\leq |\alpha|\left(\mathrm{dist}(x,x_0)\right)^{\alpha}=|\alpha| w(x).$$
The proofs of the two remaining statements being trivial are skipped.
\end{proof}

\section{Proofs of the main resuts}
\label{sec:4}

\begin{proof}[proof of Theorem~\ref{thm:2.1}] We assume fist that $L$ and $\omega$ are general and the condition (\ref{2.2.1}) is satisfied. 
Due to the convenience of the reader we divide the proof is into some steps.\\
\textbf{Step 1.} \textit{It is sufficient to prove Theorem~\ref{thm:2.1} for a single summand $w=\delta^\alpha,$ where $\alpha\in[0,1/2).$}\\
 We aim to verify that if (\ref{2.3}) is valid for the weights $w_i$, $i=1,\dots,k,$ with the same constant $C,$ then it is valid for the sum $w=\sum_{i=1}^kw_i$ with the new constant $\overline{C}=kC.$ We have by the Cauchy-Schwartz inequality that
\begin{align}
\label{4.1}
\|w\nabla u\|_{L^2(\Omega)}^2&=\int_{\Omega}w^2|\nabla u|^2dx\\ \nonumber
&=\int_{\Omega}\left(\sum_{i=1}^kw_i\right)^2|\nabla u|^2dx\\ \nonumber
&\leq k\int_{\Omega}\sum_{i=1}^kw_i^2|\nabla u|^2dx\\ \nonumber
&=k\sum_{i=1}^k\|w_i\nabla u\|_{L^2(\Omega)}^2.
\end{align}
In the other hand we have the obvious inequality
\begin{align*}
\frac{\|wu\|_{L^2(\Omega)}\cdot\|wu_{x_1}\|_{L^2(\Omega)}}{h}&+\|wu_{x_1}\|_{L^2(\Omega)}^2\\
&\geq \frac{\|w_iu\|_{L^2(\Omega)}\cdot\|w_iu_{x_1}\|_{L^2(\Omega)}}{h}+\|w_iu_{x_1}\|_{L^2(\Omega)}^2,
\end{align*}
thus combining it with (\ref{4.1}), the validity of the statement follows. Thus without loss of generality, we will assume that in what follows in the proof, the weight $w$ is a single summand, i.e., $w=\delta^\alpha,$ where $0\leq\alpha<\beta.$ We also set $\gamma=2\alpha.$\\
\textbf{Step 2.} \textit{Assume the total number of $l$ dimensional subfaces of $\Gamma_1$ is $N_l$, for $l=0,1,\dots,d,$ and denote them by $F_1^l,F_2^l,\dots,F_{N_l}^l.$ Then there exists open connected disjoint subsets $\omega_s^l$ of $\omega$ for $l=0,1,\dots,d,$ and $s=1,\dots,N_l,$ with the following properties:
\begin{itemize}
\item[1.] The function $\delta(x)=\mathrm{dist}(x,\Gamma_1)$ satisfies $\delta(x)=\mathrm{dist}(x,F_s^l)$ if $x\in\omega_s^l$ for $l=0,1,\dots,d,$ and $s=1,\dots,N_l,$
\item[2.]
$$|\omega\setminus(\cup_{l=0}^{d}\cup_{s=1}^{N_l}\omega_s^l)|=0,$$
where $|\cdot|$ stands for the $n-1$ dimensional Lebesgue measure.
\item[3.] The estimate holds
\begin{equation}
\label{4.2}
\sum_{i,j=1}^n a_{ij}\frac{\partial^2\delta^\gamma}{\partial x_i\partial x_j}\leq n\Lambda\gamma\delta^{\gamma-2},
\end{equation}
for any $l=0,1,\dots,d,$ $s=1,\dots,N_l,$ and $x\in\omega_s^l.$
\end{itemize}
}
Observe that upon rotation and translation of the coordinate system, the distances function $\delta(x)=\mathrm{dist}(x,H_i)$ from an $i$ dimensional hyperplane $H_i$ in $R^n$ is given by
$$
\delta(x)=(x_1^2+x_2^2+\ldots+x_{n-i}^2)^{1/2}.
$$
Thus the equality of the distance of the points $x\in\mathbb R^n$ from two $s_1$ and $s_2$ dimensional hyperplanes gives a hypersurface of $n-$dimensional measure zero. Next, it is clear that the distance function $\delta$ from the simplex $\Gamma_1$ is locally the distance function from one of the subfaces of $F_i^j,$ i.e., the domain $\omega$ can be partitioned into open domains $\omega_s^l,$ $l=0,1,\dots,d$ and $i=1,2,\dots,N_l$ modulo zero measure, such that in each of the sets $\omega_s^l,$ the distance $\delta$ is the distance from the $F_s^l$ subface of the simplex $\Gamma_1.$ Assume now $l=0,1,\dots,d$ and $1\leq s\leq N_l$ are fixed. Given any point $x\in\omega_s^l,$ we have $\delta(x)=\mathrm{dist}(x,F_s^l).$ We aim next to prove a Hardy-like estimate for the weighted norms of $u$ and $\nabla u$ that holds in $\omega_s^l.$ To that end we can assume without loss of generality that $0\in F_s^l $ (upon a translation of the coordinates). Assume that $F_s^l$ is $n-m$ dimensional, then we can rotate the coordinate system to put $F_s^l$ in the hyperplane $\{x_1=x_2=\ldots =x_m=0\},$ thus we get
\begin{equation}
\label{4.3}
\delta(x)=\left(\sum_{k=1}^my_k^2\right)^{1/2},
\end{equation}
where $(y_1,y_2,\dots,y_n)=B(x_1,x_2,\dots,x_n)^T$ and $B\in SO(n)$ is a rotation.
We can calculate for any $1\leq i,j\leq n,$
$$\frac{\partial\delta^\gamma}{\partial x_i}=\gamma\delta^{\gamma-2}\sum_{k=1}^{m} b_{ki}y_k,$$
and
\begin{equation}
\label{4.4}
\frac{\partial^2\delta^\gamma}{\partial x_i\partial x_j}=\gamma(\gamma-2)\delta^{\gamma-4}\sum_{k=1}^m b_{ki}y_k\sum_{l=1}^m b_{lj}y_l+
\gamma\delta^{\gamma-2}\sum_{k=1}^mb_{ki}b_{kj}.
\end{equation}
Consequently
\begin{equation}
\label{4.5}
\sum_{i,j=1}^n a_{ij}\frac{\partial^2\delta^\gamma}{\partial x_i\partial x_j}
=\gamma(\gamma-2)\delta^{\gamma-4}\sum_{i,j=1}^n a_{ij}\sum_{k=1}^m b_{ki}y_k\sum_{l=1}^m b_{lj}y_l
+\gamma\delta^{\gamma-2}\sum_{i,j=1}^n a_{ij}\sum_{k=1}^mb_{ki}b_{kj}.
\end{equation}
Using the fact that the matrix $A$ is positive definite and $\gamma\in[0,1),$ we get that the first summand in (\ref{4.5}) is nonpositive, thus we obtain the estimate
\begin{equation}
\label{4.6}
\sum_{i,j=1}^n a_{ij}\frac{\partial^2\delta^\gamma}{\partial x_i\partial x_j}\leq \gamma\delta^{\gamma-2}\sum_{i,j=1}^n a_{ij}\sum_{k=1}^mb_{ki}b_{kj}
\end{equation}
The condition (\ref{2.1}) and the orthogonality of $B$ gives the bound
$$\sum_{i,j=1}^n a_{ij}b_{ki}b_{kj}\leq \Lambda\sum_{i=1}^nb_{ki}^2\leq \Lambda,\quad\text{for all}\quad 1\leq k\leq m,$$
thus we derive from (\ref{4.6}) the bound
\begin{equation}
\label{4.7}
\sum_{i,j=1}^n a_{ij}\frac{\partial^2\delta^\gamma}{\partial x_i\partial x_j}\leq m\Lambda\gamma\delta^{\gamma-2}
\leq n\Lambda\gamma\delta^{\gamma-2} ,
\end{equation}
which is (\ref{4.2}), and Step 2 is done.\\
\textbf{Step 3.} \textit{For any $0\leq l\leq d$ and $1\leq s\leq N_l,$ the Hardy-like estimate holds:}
\begin{equation}
\label{4.8}
\int_{\omega_s^l}\delta^{\gamma-2}u^2\leq \frac{4}{(\gamma-1)^2}\int_{\omega_s^l}\delta^\gamma|\nabla u|^2.
\end{equation}
Observe that if $a>0$, $f\colon[0,a]\to \mathbb R$ is absolutely continuous with $f(0)=f(a)=0,$ then we have integrating by parts,
\begin{align*}
\int_0^a t^{\gamma-2}f^2(t)dt&=\frac{1}{\gamma-1}\int_0^a f^2(t)dt^{\gamma-1}\\
&=\frac{2}{1-\gamma}\int_0^a t^{\gamma-1}f(t)f'(t)dt,
\end{align*}
thus we have by the Schwartz inequality,
$$\int_0^a t^{\gamma-2}f^2(t)dt\leq \frac{2}{1-\gamma}\left(\int_0^a t^{\gamma-2}f^2(t)dt\right)^{1/2}
\left(\int_0^a t^\gamma f'^2(t)dt\right)^{1/2},$$
consequently we get
\begin{equation}
\label{4.9}
\int_0^a t^{\gamma-2}f^2(t)dt\leq \frac{4}{(1-\gamma)^2}\int_0^a t^\gamma f'^2(t)dt.
\end{equation}
Next we fix a point $x\in F_s^l.$ It is clear that there exists a cone $V_x$ with an apex at $x$ such that for each point $y\in V_x,$ one has $\delta(y)=|y-x|.$ Moreover, $V_{x_1}\cap V_{x_2}=\emptyset$ if $x_1,x_2\in F_s^l,$ such that $x_1\neq x_2,$ and $\cup_{x\in F_s^l}V_x=\omega_s^l.$ Take now any pooint $y\in V_x$, then by the assumption, the ray $l_x$ starting at $x$ and going through $y$ meets the boundary of $\omega$ second time at $z\in\Gamma_2.$ It is clear that one has $z\in \partial V_x$ and $x\in\partial\omega_s^l$ as well. Denote $f(t)=u(x+t(z-x))\colon[0,1]\to\mathbb R,$
which is clearly absolutely continuous. An application of (\ref{4.9}) to $f$ gives
\begin{equation}
\label{4.10}
\int_0^1 t^{\gamma-2}u^2(x+t(z-x))dt\leq \frac{4}{(1-\gamma)^2}\int_0^1 t^\gamma |z-x|^2|\nabla u(x+t(z-x))|^2dt.
\end{equation}
Denoting $y_t=x+t(z-x)\in V_x$ for $t\in(0,1)$ and noting that $\delta(y_t)=|y_t-x|=t|z-x|=$ we obtain from $(\ref{4.10})$ the following segmental integral estimate
\begin{equation}
\label{4.11}
\int_x^z \delta^{\gamma-2}u^2(y_t)dt\leq \frac{4}{(1-\gamma)^2}\int_x^z \delta^\gamma |\nabla u(y_t)|^2dt.
\end{equation}
By integrating (\ref{4.11}) over all directions from $x$ to $z$ in $V_x$ and then integrating the obtained estimtes over all $x\in F_s^l,$ we arrive at (\ref{4.8}). Before starting the last step of the proof, first observe that by Step 2, the estimate (\ref{4.2}) holds a.e. in $\Omega$, and second, by summing up the bounds (\ref{4.8}) over the indeces $l=0,1,\dots, d$ and $s=1,2,\dots, N_l,$ we get the analogous estimate for $\omega:$
\begin{equation}
\label{4.12}
\int_{\omega}\delta^{\gamma-2}u^2\leq \frac{4}{(\gamma-1)^2}\int_{\omega}\delta^\gamma|\nabla u|^2.
\end{equation}
\textbf{Step 4.} \textit{In the last step we conclude the proof the Theorem~\ref{thm:2.1}}.
For any $t\in [0,h/2]$ denote $\Omega_t=(h/2-t,h/2+t)\times \omega$ and $\Omega_t'=(0,t)\times\omega.$ Recalling that $\gamma=2\alpha,$ we have by integration by parts and using the condition (\ref{2.1}), that
\begin{align}
\label{4.13}
\lambda\int_{\Omega_t} & \delta^\gamma|\nabla u|^2dx\leq \int_{\Omega_t}\delta^\gamma\sum_{i,j=1}^na_{ij}u_{x_i}u_{x_j}dx\\ \nonumber
&=-\int_{\Omega_t}\delta^\gamma u\sum_{i=1}^n\sum_{j=2}^na_{ij}u_{x_ix_j}dx-
\int_{\Omega_t}u\sum_{i=1}^n\sum_{j=2}^na_{ij}u_{x_i}\frac{\partial }{\partial x_j}(\delta^\gamma)dx
-\int_{\Omega_t}\delta^\gamma u\sum_{j=2}^na_{1j}u_{x_1x_j}dx\\ \nonumber
&-\int_{\Omega_t}u\sum_{j=2}^na_{1j}u_{x_1}\frac{\partial }{\partial x_j}(\delta^\gamma)dx
-\int_{\Omega_t}a_{11}u\delta^\gamma u_{x_1x_1}dx
+a_{11}\int_{\omega} \left[\left(u\delta^\gamma u_{x_1}\right)|_{x_1=h/2-t}^{x_1=h/2+t}\right]dx'\\ \nonumber
&=-\int_{\Omega_t}u\delta^\gamma L(u)
-2a_{12}\int_{\Omega_t}\delta^{\gamma-1}uu_{x_1}dx
-\int_{\Omega_t}u\sum_{i,j=2}^na_{ij}u_{x_i}\frac{\partial }{\partial x_j}(\delta^\gamma)dx \\ \nonumber
&+a_{11}\int_{\omega} \left[\left(u\delta^\gamma u_{x_1}\right)|_{x_1=h/2-t}^{x_1=h/2+t}\right]dx'\\ \nonumber
&=-2a_{12}\gamma\int_{\Omega_t}\delta^{\gamma-1}uu_{x_1}dx
-\int_{\Omega_t}u\sum_{i,j=2}^na_{ij}u_{x_i}\frac{\partial }{\partial x_j}(\delta^\gamma)dx
+a_{11}\int_{\omega} \left[\left(u\delta^\gamma u_{x_1}\right)|_{x_1=h/2-t}^{x_1=h/2+t}\right]dx'.
\end{align}
For the second summand we have integrating by parts and using the estmates (\ref{4.2}) and (\ref{4.12}), that
\begin{align}
\label{4.14}
\int_{\Omega_t}u\sum_{i,j=2}^na_{ij}u_{x_i}\frac{\partial }{\partial x_j}(\delta^\gamma)dx
&=-\frac{1}{2}\int_{\Omega_t}u^2\sum_{i,j=2}^na_{ij}\frac{\partial^2(\delta^\gamma) }{\partial x_i\partial x_j}dx\\ \nonumber
&\geq -\frac{n\Lambda\gamma}{2}\int_{\Omega_t}\delta^{\gamma-2}u^2dx\\ \nonumber
&\geq -\frac{2n\Lambda\gamma}{(1-\gamma)^2}\int_{\Omega_t}\delta^{\gamma}|\nabla u|^2dx.
\end{align}
For the first summand we have by the Schwartz inequality and by (\ref{4.12}),
\begin{align}
\label{4.15}
\left|2a_{12}\gamma\int_{\Omega_t}\delta^{\gamma-1}uu_{x_1}dx\right|&\leq
2|a_{12}|\gamma\left(\int_{\Omega_t}\delta^{\gamma-2}u^2dx\right)^{1/2}\left(\int_{\Omega_t}\delta^{\gamma}|u_{x_1}|^2dx\right)^{1/2}\\ \nonumber
&\leq \frac{4|a_{12}|\gamma}{1-\gamma}\left(\int_{\Omega_t}\delta^{\gamma}|\nabla u|^2dx\right)^{1/2}\left(\int_{\Omega_t}\delta^{\gamma}|u_{x_1}|^2dx\right)^{1/2}.
\end{align}
 Combining (\ref{4.13}), (\ref{4.14}) and (\ref{4.15}) we establish the bound
\begin{align}
\label{4.16}
\left(\lambda-\frac{2n\Lambda\gamma}{(1-\gamma)^2}\right)\int_{\Omega_t}\delta^{\gamma}|\nabla u|^2dx&\leq
\Lambda\int_{\Omega_t} \left[\left(u\delta^\gamma u_{x_1}\right)|_{x_2=h/2-t}^{x_2=h/2+t}\right]dx'\\ \nonumber
&+\frac{4\Lambda\gamma}{1-\gamma}\left(\int_{\Omega_t}\delta^{\gamma}|\nabla u|^2dx\right)^{1/2}\left(\int_{\Omega_t}\delta^{\gamma}|u_{x_1}|^2dx\right)^{1/2}
\end{align}
Next we integrate (\ref{4.11}) in $t$ over the interval $(0,h/2)$ and utilize the Schwartz and the Cauchy-Schwartz inequalities to get under the condition (\ref{2.2.1}) the bound
\begin{equation}
\label{4.17}
\int_0^{h/2}dt\int_{\Omega_t}\delta^{\gamma}|\nabla u|^2dx\leq C\left(\int_{\Omega}\delta^\gamma u^2dx\right)^{1/2}
\left(\int_{\Omega}\delta^\gamma |u_{x_1}|^2dx\right)^{1/2}+Ch\int_{\Omega}\delta^\gamma |u_{x_1}|^2dx.
\end{equation}
Observe, that the function $\int_{\Omega_t}\delta^{\gamma}|\nabla u|^2dx$ increases in $t,$ thus we get from (\ref{4.17}) the estimate
\begin{equation}
\label{4.18}
\int_{\Omega_{h/4}}\delta^{\gamma}|\nabla u|^2dx\leq \frac{C}{h}\left(\int_{\Omega}\delta^\gamma u^2dx\right)^{1/2}
\left(\int_{\Omega}\delta^\gamma |u_{x_1}|^2dx\right)^{1/2}+C\int_{\Omega}\delta^\gamma |u_{x_1}|^2dx.
\end{equation}
Having (\ref{4.18}) in hand, it remains to estimate the quantity $\int_{\Omega_{h/4}'}\delta^\gamma|\nabla u|^2dx.$ Next we recall the following Hardy-like estimate established by Kondratiev and Oleinik [\ref{bib:Kon.Ole.1},\ref{bib:Kon.Ole.2},\ref{bib:Ole.Sha.Yos.}].
\begin{lemma}
\label{lem:4.1}
Assume $a>0$ and $f\colon[0,a]\to\mathbb R$ is absolutely countinuous. Then there holds:
$$\int_0^{a/2}f^2(t)dt\leq 4\int_{a/2}^af^2(t)dt+4\int_0^a t^2|f'(t)|^2dt.$$
\end{lemma}
We fix any index $2\leq k\leq n,$ any point $x'\in\omega$ and apply Lemma~\ref{4.1} to the function $f(t)=\delta^{\alpha} u_{x_k}(t,x')$ over the interval with the endpoints $(0,x')$ and $(h/2,x').$ We have that
$$\int_0^{h/4}\delta^\gamma |u_{x_k}(t,x')|^2dt\leq 4\int_{h/4}^{h/2}\delta^\gamma |u_{x_k}(t,x')|^2dt
+4\int_0^{h/2}\delta^\gamma t^2 |u_{x_kx_1}(t,x')|^2dt.$$
Summing the above inequalities over $2\leq k\leq n,$ and addind the missing summand on the left hand side we arrive at
$$\int_0^{h/4}\delta^\gamma |\nabla u(t,x')|^2dt\leq 4\int_{h/4}^{h/2}\delta^\gamma |\nabla u(t,x')|^2dt+\int_0^{h/4}\delta^\gamma |u_{x_1}(t,x')|^2dt
+4\int_0^{h/2}\delta^\gamma t^2|\nabla u_{x_1}(t,x')|^2dt.$$
Upon integrating the last estimate in $x'$ over $\omega$ we obtain
\begin{equation}
\label{4.19}
\int_{\Omega_{h/4}'}\delta^\gamma |\nabla u|^2\leq 4\int_{\Omega_{h/4}}\delta^\gamma |\nabla u|^2+
\int_{\Omega_{h/4}'}\delta^\gamma |u_{x_1}|^2+4\int_{\Omega_{h/2}'}\delta^\gamma x_1^2|\nabla u_{x_1}|^2.
\end{equation}
From the boundary conditions $u=0$ on $[0,h]\times\partial \omega$ we have that $u_{x_1}=0$ on $[0,h]\times\partial \omega$, thus we can apply Lemma~\ref{lem:3.1} to the function $u_{x_1}$ with the weight $w=\delta^\alpha$ in $\Omega$ to get
$$\int_{\Omega}\delta^\gamma |\nabla u_{x_1}|^2\leq C\int_{\Omega}\delta^\gamma |u_{x_1}|^2+
C\int_{\Omega}\delta^{2\gamma} |\delta^2L(u_{x_1})|^2.$$
Observe that by differentiating the equality $L(u)=0$ we get $L(u_{x_1})=0$ in $\Omega,$ thus the last estimate simplifies to
\begin{equation}
\label{4.20}
\int_{\Omega}\delta^\gamma |\delta \nabla u_{x_1}|^2\leq C\int_{\Omega}\delta^\gamma |u_{x_1}|^2.
\end{equation}
Note that we have on the other hand
\begin{equation}
\label{4.21}
\int_{\Omega_{h/2}'}\delta^\gamma x_1^2|\nabla u_{x_1}|^2\leq \int_{\Omega}\delta^\gamma |\delta \nabla u_{x_1}|^2,
\end{equation}
thus combining the estimates (\ref{4.19})-(\ref{4.21}) we discover
\begin{equation}
\label{4.22}
\int_{\Omega_{h/4}'}\delta^\gamma |\nabla u|^2\leq 4\int_{\Omega_{h/4}}\delta^\gamma |\nabla u|^2+
C\int_{\Omega}\delta^\gamma |u_{x_1}|^2.
\end{equation}
It is clear that a similar estimate holds also for the slice of $\Omega$ that is obtained from $\Omega_{h/4}'$ by mirror symmetry about the plane $x_1=h/2.$ Thus putting togehter (\ref{4.18}) and (\ref{4.22}) we establish the estimate
\begin{equation}
\label{4.23}
\int_{\Omega}\delta^\gamma |\nabla u|^2\leq \frac{C}{h}\left(\int_{\Omega}\delta^\gamma u^2dx\right)^{1/2}
\left(\int_{\Omega}\delta^\gamma |u_{x_1}|^2dx\right)^{1/2}+C\int_{\Omega}\delta^\gamma |u_{x_1}|^2,
\end{equation}
which yields (\ref{2.3}). The proof in the case of the peresence of the cndition (\ref{2.2.1}) is complete.\\ 
Assume now $n=2$ and $L=\Delta.$ It is straightforward to check that in the case $0\leq \alpha\leq 1/2$ there is no need to prove the estimates like (\ref{4.7}), (\ref{4.8}) and (\ref{4.12}) (also they make no sense for $\alpha=1/2$) as they are basically needed to estimate the firs and second summands in the last line of (\ref{4.13}). Instead, if $L=\Delta$ then apparently the first summand in (\ref{4.13}) vanishes and if $n=2,$ then we have for the second summand in (\ref{4.13}) integrating by parts, that (as then $\delta=x_2$)
$$
-\int_{\Omega_t}u\sum_{i,j=2}^na_{ij}u_{x_i}\frac{\partial }{\partial x_j}(\delta^\gamma)dx
=\frac{\gamma(\gamma-1)}{2}\int_{\Omega_t}u^2x_2^{\gamma-2}dx\leq 0,
$$
which is to replace the estimate (\ref{4.14}).
\end{proof}

\begin{proof}[Proof of Theorem~\ref{thm:2.2}]
We adopt the main strategy of proving Korn or a related inequality. Namely, we firs assume without loss of generality that $\BU$ is smooth up to the boundary of $R$ and the consider the harmonic part $\tilde u$ of $u,$ i.e., assume $\tilde u\in H^1(R)$ is the unique solution of the Dirichlet boundary value problem
\begin{equation}
\label{4.24}
\begin{cases}
\Delta\tilde u=0 & \ \text{in} \ R,\\
\tilde u=u & \ \text{on} \ \partial R.
\end{cases}
\end{equation}
We can calculate
$$\Delta(u-\tilde u)=\Delta u=(e_{11}(\BU)-e_{11}(\BU))_x+(2e_{12}(\BU))_y,$$
thus we can evaluate
\begin{align}
\label{4.25}
\int_{R}y^\gamma|\nabla(u-\tilde u)|^2&=-\int_{R}\left[(u-\tilde u)(y^\gamma(u-\tilde u)_{x})_{x}
+(u-\tilde u)(y^\gamma(u-\tilde u)_{y})_{y}\right]\\ \nonumber
&=-\int_{R}(u-\tilde u)[y^\gamma\Delta(u-\tilde u)+\gamma y^{\gamma-1}(u-\tilde u)_{y}]\\ \nonumber
&=I_1+I_2,
\end{align}
where
\begin{align}
\label{4.26}
I_1&=-\int_{R}y^\gamma(u-\tilde u)[(e_{11}(\BU)-e_{22}(\BU))_{x}+2(e_{12}(\BU))_{y}]\\ \nonumber
I_2&=-\int_{R}\gamma y^{\gamma-1}(u-\tilde u)(u-\tilde u)_{y}.
\end{align}
Due to the fact that $u-\tilde u$ vanishes on $\partial R,$ we can calculate by integration by parts that
\begin{equation}
\label{4.27}
I_2=\frac{\gamma(\gamma-1)}{2}\int_R y^{\gamma-2}(u-\tilde u)^2\leq 0.
\end{equation}
For the first summand we have again integrating by parts
\begin{equation}
\label{4.28}
I_1=\int_{R}y^\gamma(u-\tilde u)_x(e_{11}(\BU)-e_{22}(\BU))
+2\int_{R}y^\gamma(u-\tilde u)_ye_{12}(\BU)
+2\gamma \int_{R}y^{\gamma-1}(u-\tilde u)e_{12}(\BU).
\end{equation}
Consequently we obtain by the Schwartz inequality
\begin{equation}
\label{4.29}
I_1\leq 4\|y^\alpha\nabla(u-\tilde u)\|\cdot\|y^\alpha e(\BU)\|+2\gamma\|y^{\alpha-1}(u-\tilde u)\|\cdot \|y^\alpha e(\BU)\|.
\end{equation}
Note, that we derived (\ref{4.12}) under the sole condition on the function on $u$ that it vanishes on $\partial\omega,$ thus (\ref{4.12}) holds for the function $u-\tilde u$ in the domain $R$ with the weight $w=y^\alpha,$ i.e., we have the estimate
$$\|y^{\alpha-1}(u-\tilde u)\|\leq \frac{2}{1-\gamma}\|y^\alpha\nabla(u-\tilde u)\|,$$
which together with (\ref{4.29}) gives the bound
\begin{equation}
\label{4.30}
I_1\leq \frac{4}{1-\gamma}\|y^\alpha\nabla(u-\tilde u)\|\cdot \|y^\alpha e(\BU)\|.
\end{equation}
Combining (\ref{4.25}), (\ref{4.27}) and (\ref{4.30}) we arrive at the bound
\begin{equation}
\label{4.31}
\|y^\alpha\nabla(u-\tilde u)\|\leq \frac{4}{1-\gamma}\|y^\alpha e(\BU)\|.
\end{equation}
Also, we have by the Poincar\'e inequality (not with the best constant) in the $x$ direction that
\begin{equation}
\label{4.32}
\|y^\alpha(u-\tilde u)\|\leq h \|y^\alpha(u-\tilde u)_x\|\leq \|y^\alpha\nabla(u-\tilde u)\|\leq \frac{4h}{1-\gamma}\|y^\alpha e(\BU)\|.
\end{equation}
We apply Theorem~\ref{thm:2.1} to the function $\tilde u$ in the domain $\Omega=R.$ By virtue of the triangle inequality and the estimates (\ref{4.31}) and (\ref{4.32}) we can develop the following chain of estimates:
\begin{align}
\label{4.33}
\|y^\alpha u_y\|^2&\leq 2\|y^\alpha (u-\tilde u)_y\|^2+2\|y^\alpha\tilde u_y\|^2\\ \nonumber
&\leq \frac{8}{1-\gamma}\|y^\alpha e(\BU)\|^2+\frac{C}{h}\|y^\alpha\tilde u_x\|\cdot \|y^\alpha\tilde u\|+C\|y^\alpha\tilde u_x\|^2\\ \nonumber
&\leq C\|y^\alpha e(\BU)\|^2+\frac{C}{h}(\|y^\alpha u_x\|+\|y^\alpha \nabla (u-\tilde u)\|)(\|y^\alpha u\|+\|y^\alpha (u-\tilde u)\|\\ \nonumber
&+C\|y^\alpha u_x\|^2+C\|y^\alpha \nabla (u-\tilde u)\|^2\\ \nonumber
&\leq \frac{C}{h}\|e(\BU)\|\cdot\|y^\alpha u\|+C\|e(\BU)\|^2.
\end{align}
The norm $\|y^\alpha v_x\|$ can then be estimated in terms of $\|y^\alpha u_y\|$ as above and $e_{12}(\BU)$ by the triangle inequality. The proof of the theorem is complete.

\end{proof}

\begin{proof}[Proof of Theorem~\ref{thm:2.3}]. 
The Ansatz realizing the asymptotics of $h$ for both inequalities (\ref{3.2}) and (\ref{3.3}) comes from the papers [\ref{bib:Gra.Har.1},\ref{bib:Harutyunyan1}]. In the case when $L=\Delta$ and $\omega=(a,b)\times\tilde\omega$ we can use the Ansatz 
$$u(x)=\cosh\left(\frac{\pi}{b-a}\left(x_1-\frac{h}{2}\right)\right)\sin\left(\frac{\pi x_2}{b-a}\right),$$
for the estimate (\ref{3.2}). For (\ref{3.3}) we use the Ansatz 
$$U=\left(f\left(\frac{y}{h^\alpha}\right),-\frac{x}{h^\alpha}f\left(\frac{y}{h^\alpha}\right)\right),$$
where $f$ is a smooth function supported on $(a,b)$ and $\alpha\in[0,1/2].$ 
The calculation for both cases is straightforward and is omitted here.  
\end{proof}

\end{document}

%% file: mymacros.tex
\usepackage{amssymb,bm}
\usepackage{amsmath}
\usepackage{amsthm}
\usepackage{graphicx}
\usepackage[active]{srcltx} 
\usepackage{hyperref}
\hypersetup{pdfborder=0 0 0}

\newtheorem{theorem}{{\sc Theorem}}[section]

\newtheorem{lemma}[theorem]{{\sc Lemma}}

\def\XXint#1#2#3{{\setbox0=\hbox{$#1{#2#3}{\int}$ }
\vcenter{\hbox{$#2#3$ }}\kern-.6\wd0}}

\bmdefine\BGa{\alpha}
\bmdefine\BGb{\beta}
\bmdefine\BGd{\delta}
\bmdefine\BGe{\epsilon}
\bmdefine\BGve{\varepsilon}
\bmdefine\BGf{\phi}
\bmdefine\BGvf{\varphi}
\bmdefine\BGg{\gamma}
\bmdefine\BGc{\chi}
\bmdefine\BGi{\iota}
\bmdefine\BGk{\kappa}
\bmdefine\BGl{\lambda}
\bmdefine\BGn{\eta}
\bmdefine\BGm{\mu}
\bmdefine\BGv{\nu}
\bmdefine\BGp{\pi}
\bmdefine\BGth{\theta}
\bmdefine\BGvth{\vartheta}
\bmdefine\BGr{\rho}
\bmdefine\BGvr{\varrho}
\bmdefine\BGs{\sigma}
\bmdefine\BGvs{\varsigma}
\bmdefine\BGt{\tau}
\bmdefine\BGj{\tau}
\bmdefine\BGu{\upsilon}
\bmdefine\BGo{\omega}
\bmdefine\BGx{\xi}
\bmdefine\BGy{\psi}
\bmdefine\BGz{\zeta}
\bmdefine\BGD{\Delta}
\bmdefine\BGF{\Phi}
\bmdefine\BGG{\Gamma}
\bmdefine\BGL{\Lambda}
\bmdefine\BGP{\Pi}
\bmdefine\BGT{\Theta}
\bmdefine\BGS{\Sigma}
\bmdefine\BGU{\Upsilon}
\bmdefine\BGO{\Omega}
\bmdefine\BGX{\Xi}
\bmdefine\BGY{\Psi}

\bmdefine\BCA{{\mathcal A}}
\bmdefine\BCB{{\mathcal B}}
\bmdefine\BCC{{\mathcal C}}
\bmdefine\BCD{{\mathcal D}}
\bmdefine\BCE{{\mathcal E}}
\bmdefine\BCF{{\mathcal F}}
\bmdefine\BCG{{\mathcal G}}
\bmdefine\BCH{{\mathcal H}}
\bmdefine\BCI{{\mathcal I}}
\bmdefine\BCJ{{\mathcal J}}
\bmdefine\BCK{{\mathcal K}}
\bmdefine\BCL{{\mathcal L}}
\bmdefine\BCM{{\mathcal M}}
\bmdefine\BCN{{\mathcal N}}
\bmdefine\BCO{{\mathcal O}}
\bmdefine\BCP{{\mathcal P}}
\bmdefine\BCQ{{\mathcal Q}}
\bmdefine\BCR{{\mathcal R}}
\bmdefine\BCS{{\mathcal S}}
\bmdefine\BCT{{\mathcal T}}
\bmdefine\BCU{{\mathcal U}}
\bmdefine\BCV{{\mathcal V}}
\bmdefine\BCW{{\mathcal W}}
\bmdefine\BCX{{\mathcal X}}
\bmdefine\BCY{{\mathcal Y}}
\bmdefine\BCZ{{\mathcal Z}}

\bmdefine\Bzr{ 0}
\bmdefine\Ba{ a}
\bmdefine\Bb{ b}
\bmdefine\Bc{ c}
\bmdefine\Bd{ d}
\bmdefine\Be{ e}
\bmdefine\Bf{ f}
\bmdefine\Bg{ g}
\bmdefine\Bh{ h}
\bmdefine\Bi{ i}
\bmdefine\Bj{ j}
\bmdefine\Bk{ k}
\bmdefine\Bl{ l}
\bmdefine\Bm{ m}
\bmdefine\Bn{ n}
\bmdefine\Bo{ o}
\bmdefine\Bp{ p}
\bmdefine\Bq{ q}
\bmdefine\Br{ r}
\bmdefine\Bs{ s}
\bmdefine\Bt{ t}
\bmdefine\Bu{ u}
\bmdefine\Bv{ v}
\bmdefine\Bw{ w}
\bmdefine\Bx{ x}
\bmdefine\By{ y}
\bmdefine\Bz{ z}
\bmdefine\BA{ A}
\bmdefine\BB{ B}
\bmdefine\BC{ C}
\bmdefine\BD{ D}
\bmdefine\BE{ E}
\bmdefine\BF{ F}
\bmdefine\BG{ G}
\bmdefine\BH{ H}
\bmdefine\BI{ I}
\bmdefine\BJ{ J}
\bmdefine\BK{ K}
\bmdefine\BL{ L}
\bmdefine\BM{ M}
\bmdefine\BN{ N}
\bmdefine\BO{ O}
\bmdefine\BP{ P}
\bmdefine\BQ{ Q}
\bmdefine\BR{ R}
\bmdefine\BS{ S}
\bmdefine\BT{ T}
\bmdefine\BU{ U}
\bmdefine\BV{ V}
\bmdefine\BW{ W}
\bmdefine\BX{ X}
\bmdefine\BY{ Y}
\bmdefine\BZ{ Z}

